\documentclass[11pt]{article}
\usepackage{latexsym,amssymb,amsmath,amsfonts,amsthm}
\usepackage{graphics}
\usepackage{graphicx}
\usepackage{mathrsfs}
\usepackage{subfigure}
\usepackage{color}
\usepackage{bm}
\newcommand{\R}{{\mat R}}

\newcommand{\no}{\nonumber}
\newcommand{\be}{\begin{eqnarray}}
\newcommand{\ben}{\begin{eqnarray*}}
\newcommand{\en}{\end{eqnarray}}
\newcommand{\enn}{\end{eqnarray*}}

\newcommand{\pa}{\partial}

\newcommand{\ov}{\overline}
\newcommand{\I}{{\rm Im}}

\newcommand{\G}{\Gamma}

\newcommand{\Om}{\Omega}
\newcommand{\om}{\omega}

\newcommand{\mat}{\mathbb}

\newcommand{\la}{\lambda}

\newtheorem{theorem}{Theorem}[section]
\newtheorem{lemma}[theorem]{Lemma}

\newtheorem{algorithm}{Algorithm}[section]
\newtheorem{problem}[theorem]{Problem}

\begin{document}
\renewcommand{\theequation}{\arabic{section}.\arabic{equation}}
\title{\bf A direct imaging method for inverse elastic scattering by unbounded rigid rough surfaces
}
\author{Guanghui Hu\thanks{Beijing Computational Science Research Center,
Beijing 100193, China ({\tt hu@csrc.ac.cn})}
\and Xiaoli Liu\thanks{Academy of Mathematics and Systems Science, Chinese Academy of Sciences,
Beijing 100190, China and School of Mathematical Sciences, University of Chinese
Academy of Sciences, Beijing 100049, China ({\tt liuxiaoli@amss.ac.cn})}
\and Bo Zhang\thanks{LSEC, NCMIS and Academy of Mathematics and Systems Science, Chinese Academy of
Sciences, Beijing, 100190, China and School of Mathematical Sciences, University of Chinese
Academy of Sciences, Beijing 100049, China ({\tt b.zhang@amt.ac.cn})}
\and Haiwen Zhang\thanks{NCMIS and Academy of Mathematics and Systems Science, Chinese Academy of Sciences,
Beijing 100190, China ({\tt zhanghaiwen@amss.ac.cn})}}
\date{}

\maketitle


\begin{abstract}
This paper is concerned with the inverse time-harmonic elastic scattering problem of recovering unbounded
rough surfaces in two dimensions. We assume that elastic plane waves with different directions are incident
onto a rigid rough surface in a half plane. The elastic scattered field is measured on a horizontal straight
line segment within a finite distance above the rough surface. A direct imaging algorithm is proposed to
recover the unbounded rough surface from the scattered near-field data, which involves only inner products
between the data. Numerical experiments are presented to show that the inversion scheme is
not only efficient but also accurate and robust with respect to noise.

\vspace{.2in}
\end{abstract}

\section{Introduction}
\setcounter{equation}{0}

In this paper, we study the inverse scattering problem of time-harmonic elastic
waves by an unbounded rough surface in two dimensions.
The region above the surface is filled with a homogeneous and isotropic elastic medium,
while the material below is assumed to be elastically rigid.
Our goal is to recover the unbounded interface from the scattered field measured on a horizontal
straight line segment at a finite distance above the rough surface.
Such kind of inverse problems occurs in many applications such as
geophysics, seismology and nondestructive testing \cite{BC-2005,LL-1986,S-1984}.
See Figure \ref{fig0} for an illustration of the scattering problem in a half plane.

\begin{figure}[htbp]
  \centering
    \includegraphics[width=3.5in]{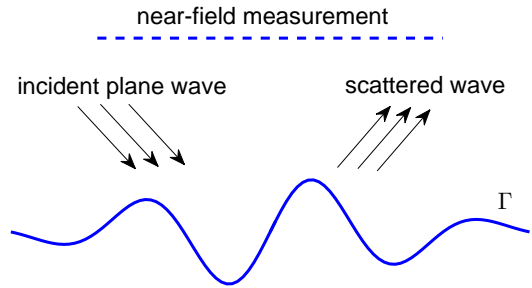}
\caption{Elastic scattering by an unbounded rigid rough surface in 2D.
}\label{fig0}
\end{figure}

Given an incident field and the rough surface, the forward problem is to determine the field distribution
of the scattered field. Uniqueness and existence of forward solutions have been investigated
in \cite{A00, Arens-2001,Arens-2002} by the integral equation method and in \cite{Hu-2012-SIAM,Hu-2015-AA}
by a variational method. These approaches extend the solvability results for grating diffraction problems
to the more challenging case of unbounded rough surfaces in elasticity. In the periodic case,
numerical methods such as the two-step optimization method \cite{Hu-2012}, the factorization method \cite{Hu-2013}
and the transformation-based near-field imaging method \cite{Zhao-2015,Zhao-2016} have been proposed to
solve the inverse elastic scattering problems for diffraction gratings.
Note that the periodic setting significantly simplifies the arguments for the case of general rough surfaces.
However, relatively little studies are carried out for inverse elastic scattering arising from global rough
surfaces. This is mainly due to the infiniteness of the scattering surface, which brings additional
difficulties not only to mathematical analysis but also to numerical computation.

In this paper, we shall propose a direct imaging method to reconstruct the unbounded rough surface from
the near-field measurement of the scattered field generated by plane elastic waves at one frequency
with multiple directions. We take inspirations from recent studies for inverse acoustic scattering problems,
for example, the orthogonality sampling method \cite{P2010}, topological derivative-based approach \cite{BBC},
the direct sampling method \cite{IJZ-2012} and the reverse time migration method \cite{CH-2013}.
Compared to iterative schemes and other sampling type methods, the features of the direct imaging method
include: (i) Capability of depicting the profile of the surface only through computing the inner products
of the measured data and a known function at each sampling point. Thus, the computational cost is very cheap.
This merit is especially important for elastic scattering problems since the computation for vectorial equations
is usually more time-consuming than scalar equations. (ii) Robustness to a fairly large amount of noise
in the measured data. This paper focuses on the unbounded rough surface identification problem
in linear elasticity and is an extension of our recent work \cite{LZZ} on near-field imaging penetrable
interfaces modelled by the Helmholtz equation to the {\em Navier} equation.
The elasticity problem takes a more complicated form than the acoustic case, due to the coexistence of
the compressional and shear waves that propagate at different speeds. We consider only impenetrable rigid
rough surfaces (on which the elastic displacement vanishes), but our method applies naturally to interfaces
with other boundary or transmission conditions.

Our imaging scheme relies essentially on a relation between the Funk-Hecke formula and the free-space Green's
tensor for the {\em Navier} equation; see (\ref{impi-1}) or (\ref{ImP}).
Motivated by this relation, we present the scattered field in the form of a superposition
of incident elastic plane waves (see Theorem \ref{mainthm} below).
This expression of the scattered field will be proven to have the same decaying property as the imaginary
part of the Green's tensor, as the sampling point moves away from the scattering surface.
This yields our indicator function using plane elastic waves with different directions in a half plane.
Numerical experiments are presented to show the effectiveness of our method.
Further, we investigate the effect of the reconstructed results from the parameters such as the incident frequency, the measurement place and the noise level. Numerics show that our imaging algorithm
is fast, accurate and very robust with respect to the noisy data.

The remainder of this paper is organized as follows. In Section \ref{sec2}, we briefly review the well-posedness
of the forward scattering problem using the integral equation method.
Section \ref{sec3} is devoted to an analysis of our imaging function and a description of
the direct imaging algorithm. Numerical experiments are carried out in Section \ref{sec4} to demonstrate the
effectiveness of the proposed approach.

\section{Well-posedness of the forward scattering problem}\label{sec2}
\setcounter{equation}{0}

In the section we review existence and uniqueness of solutions to elastic scattering from
rigid rough surfaces in two dimensions. The solvability of the second kind integral operator
established in \cite{Arens-2002} will be used to analyze our indicator function to be proposed
in Section \ref{sec3}.

Consider a one-dimensional unbounded rough surface
$\G=\{x\in \R^2~|~x_2=f(x_1),x_1\in \R\}$, where $\G$ is supposed to be smooth enough such that
$f \in BC^{1, 1}(\mathbb{R})$.
Here, $BC^{1, 1}(\mathbb{R}):=\{\varphi\in BC(\R)~|~\varphi'\in BC(\R),j=1,2\}$
under the norm $\|\varphi\|_{1,\R}:=\|\varphi\|_{\infty,\R}+\|\varphi'\|_{\infty,\R}$
and $BC(\R)$ is the set of bounded and continuous functions in $\R$.
Denote the region above $\G$ by $\Om$. Assume that $\Om$ is filled with an isotropic homogeneous elastic medium
characterized by the Lam\'e constants $\mu, \lambda$ with $\mu >0, \lambda + \mu \geq 0$. For simplicity,
we assume that the density function in $\Omega$ is normalized to be one and the region below $\Gamma$ is
a rigid elastic body. Assume that a time-harmonic plane wave (with time variation of the form
exp($-i\om t$), $\om>0$) is incident onto $\Gamma$ from $\Omega$. The incident plane wave ${\bm u}^{in}$
can be either the pressure wave
\ben
{\bm u}^{in}_p(x;{\bm d}):={\bm d}e^{ik_p{\bm d} \cdot x},
\enn
or the shear wave
\ben
{\bm u}^{in}_s(x;{\bm d}):={\bm d^\perp}e^{ik_s{\bm d} \cdot x},
\enn
where ${\bm d}=(d_1,d_2)^T \in \mathbb{S}:=\{x=(x_1,x_2)~|~|x|=1\}$ is the incident direction and
${\bm d}^\perp=(-d_2,d_1)^T$.
The compressional wave number $k_p$ and the shear wave number $k_s$ are given by
\ben
k_p={\om}/{\sqrt{\lambda+2\mu}},\quad k_s= {\om}/{\sqrt{\mu}}.
\enn
The displacement of the total field ${\bm u}=(u_1,u_2)^T$ is then governed by the \emph{Navier} equation
\begin{equation}\label{navier}
\mu \Delta {\bm u}+(\lambda + \mu)\nabla\nabla\cdot{\bm u}+\om^2{\bm u}=0\quad\text{in}\quad\Om,
\end{equation}
together with the Dirichlet boundary condition
\be\label{bc}
{\bm u}=0\quad\text{on}\quad \G.
\en
Given a curve $\Lambda \subset \mathbb{R}^2$ with the unit normal $\bf n\in \mathbb{S}$,
the \emph{generalised stress} operator ${\bm P}$ on $\Lambda$ is defined by
\ben
{\bm P}{\bm u}:=(\mu+\tilde{\mu})\frac{\pa{\bm u}}{\pa\bm n}+\tilde{\lambda}{\bm n}\nabla\cdot{\bm u}
- \tilde{\mu}{\bm n}^{\perp}\nabla^{\perp}\cdot {\bm u},\quad \nabla^{\perp}:=(-\partial_2,\partial_1).
\enn
Here, $\tilde{\mu},\tilde{\lambda}$ are real numbers satisfying $\tilde{\mu}+\tilde{\lambda}=\mu+\lambda$.
A special choice of $\tilde{\mu}$ and $\tilde{\lambda}$ with $\tilde{\mu}=\mu(\mu+\lambda)/(3\mu+\lambda)$ and
$\tilde{\lambda}=(2\mu+\lambda)(\mu+\lambda)/(3\mu+\lambda)$ will be used in this paper;
see \cite[Chapter 3]{A00} for details.

Let $\G_a:=\{x=(x_1,x_2)~|~x_2=a\}$ and $U_a:=\{x=(x_1,x_2) ~|~ x_2>a\}$. In this paper we require
the scattered field ${\bm u}^{sc}={\bm u-\bm u}^{in}$ to fulfill the
{\em Upwards Propagating Radiation Condition (UPRC)} (see \cite{Arens-2001}):
\be\label{uprc}
{\bm u}^{sc}(x)=\int_{\G_a}{\bm P}_y[\Pi_a(x,y)]\phi(y)ds(y),\quad x\in U_a
\en
for some $a>f_+:=\sup\limits_{x_1\in\R}f(x_1)$ with some $\phi \in [L^{\infty}(\G_a)]^2$.
Here, $\Pi_a(x,y)$ denotes the Green's tensor for the {\em Navier} equation in the half
plane $x_2>a$ with the homogeneous Dirichlet boundary condition on $\G_a$.
The differential operator ${\bm P}_y[\Pi_a(x,y)]$ is understood as the action of
the \emph{generalised stress} operator ${\bm P}$ to each column of
$\Pi_a(x,y)$ with respect to the argument $y$.
The explicit expression of $\Pi_a(x,y)$ can be found in \cite{Arens-2001}.
Below we formulate the forward elastic scattering problem as a boundary value problem.

\begin{problem} \label{pb1}
Given ${\bm g}\in [BC(\G)\cap H^{1/2}_{loc}(\G)]^2$, find a vector field
$\bm{u}\in [C^2(\Om)\cap C(\ov{\Omega})\cap H^1_{loc}(\Omega)]^2$ that satisfies
\begin{enumerate}
\item
the Navier equation (\ref{navier}) in $\Om$,
\item
the Dirichlet boundary condition $\bm {u = g}$ on $\G$,
\item
the vertical growth rate condition:
$
\sup\limits_{{x} \in \Omega}|x_2|^\beta|{\bm u}{(x)}|<\infty
$
for some $\beta \in \mathbb{R}$ ,
\item
the UPRC (\ref{uprc}).
\end{enumerate}
\end{problem}

Well-posedness of the forward elastic scattering of plane waves in 2D was investigated
in \cite{A00,Arens-2001,Arens-2002} using the integral equation method and in \cite{Hu-2015-AA}
using a variational aproach in weighted Sobolev spaces.
In particular, the unique solution to Problem \ref{pb1} can be written in the form of a combined
single- and double-layer potential (see \cite{Arens-2001,Arens-2002})
\ben
{\bm u}(x)=\int_{\G}\left\{{\bm P}_y[\Pi_h(x,y)]-i\eta\Pi_{h}({x,y})\right\}{\bm{\varphi}}({y})ds({y}),\quad{x}\in\Om,
\enn
where $h<\inf\limits_{x_1\in\R}f(x_1)$, $\eta$ is a complex number satisfying $\Re(\eta)>0$ and the density
function ${\bm\varphi}\in[BC(\G)\cap H^{1/2}_{loc}(\G)]^2$ is the unique solution to the boundary integral equation
\ben
({\bm I}+{\bm D} - i \eta {\bm S}){\bm\varphi}(y)=2{\bm u}^{in}(y), \quad y \in \G.
\enn
Note that the boundary integral operators ${\bm S}$ and ${\bm D}$ are defined, respectively, as
\ben
{\bm {S \varphi}}(y):= 2 \int_{\G}\Pi_{h}(y,\xi) {\bm \varphi}(\xi)d\xi, \quad
{\bm {D \varphi}}(y):= 2 \int_{\G}{\bm P}_y[\Pi_h(y,\xi)] {\bm \varphi}(\xi)d\xi.
\enn
It was verified in \cite{Arens-2002} that the operator ${\bm I}+{\bm D}-i\eta{\bm S}$ is bijective on $[BC(\G)]^2$.
Further, it holds that
\ben
\|({\bm I}+{\bm D} - i \eta {\bm S})^{-1}\|< \infty.
\enn
We summarize the well-posedness of Problem \ref{pb1} in the following theorem.

\begin{theorem}\label{thmwp} (see \cite[Theorem 5.24]{A00})
For any Dirichlet data ${\bm g}\in [BC(\G)\cap H^{1/2}_{loc}(\G)]^2 $, there exists a unique solution
${\bm u}\in [C^2(\Om)\cap C(\ov{\Om})\cap H^1_{loc}(\Om)]^2$ to Problem \ref{pb1}, which depends continuously
on $\|{\bm g}\|_{\infty;\G}$, uniformly in $[C(\ov\Om\setminus U_a)]^2$ for any $a>f_+$.
\end{theorem}

\section{The direct imaging method}\label{sec3}
\setcounter{equation}{0}

Introduce the notation
\ben
&& \G_{a,A}:=\{x\in \Om~|~x_2=a,|x_1|\leq A\},\, a>f_+.\\
&&\mathbb{S}_+:=\left\{{\bm d}=(d_1,d_2)^T~|~|{\bm d}|=1,d_2>0\right\},\\
&&\mathbb{S}_-:=\left\{{\bm d}=(d_1,d_2)^T~|~|{\bm d}|=1,d_2<0\right\}.
\enn
The purpose of this section is to propose a direct imaging scheme for
determining  $\G$ from
the scattered near-field data $\left\{{\bm u}^{sc}(x;{\bm d})~|~x\in\G_{a,A},{\bm d}\in \mathbb{S}_-\right\}$
incited by elastic plane waves ${\bm u}_p^{sc}(x;{\bm d})$ and ${\bm u}_s^{sc}(x;{\bm d})$.

We begin with the free-field Green's tensor $\Pi({x,y})$ for the two-dimensional {\em Navier} equation,
given by
\be \label{Pi}
\Pi({x,z}):=\frac{1}{\mu} {\bm I }\, \varPhi_{k_s}({x,z})+
\frac{1}{\omega^2} \nabla^T_x \nabla_x \left(\varPhi_{k_s}({x,z})-\varPhi_{k_p}({x,z}) \right),
\en
with ${x,z} \in \mathbb{R}^2$ and ${x \neq z}$. Here, ${\bm I }$ denotes the 2-by-2 unit matrix and  the scalar function $\varPhi_{k}({x,z})$ is the fundamental solution to the two-dimensional {\em Helmholtz} equation given by
\be\label{phii}
\varPhi_k({x,z})=\frac{i}{4}H^{(1)}_0(k|{x-z}|), \quad {x \not= z},
\en
where $H^{(1)}_0:=J_0+i Y_0$ is the Hankel function of the first kind of order zero. The functions $J_n$ and $Y_n$ are the Bessel and Neumann functions of order $n$, respectively.
To derive our indicator function, we need
 the following {\em Funk-Hecke} formula (see e.g., \cite{colton-kress-2013}):
\begin{lemma}
For any $k>0$, we have
\ben
\frac{1}{2\pi}\int_{\mathbb{S}}e^{ik({ x-z})\cdot {\bm d}}ds({\bm d})= J_0(k|{x-z}|),
\quad x,z\in \R^2.
\enn
\end{lemma}
Combining the above lemma with (\ref{phii}), we obtain
\be\label{3}
\text{Im}\left(\varPhi_k(x,z)\right)=\frac{1}{4}J_0(k|{x-z|})
=\frac{1}{8\pi}\int_{\mathbb{S}}e^{ik({x-z})\cdot {\bm d}}ds({\bm d}).
\en
Taking the imaginary part of (\ref{Pi}) and using (\ref{3}) yield
\be\label{impi-1}
\text{Im}~\Pi({x,z})=\frac{1}{8\pi}\left[\frac{1}{\la+2\mu}\int_{\mathbb{S}}{\bm d}\otimes{\bm d}e^{ik_p({x-z})\cdot
{\bm d}}ds({\bm d})+\frac{1}{\mu}\int_{\mathbb{S}}({\bm I}-{\bm d}\otimes {\bm d})e^{ik_s({x-z})\cdot
{\bm d}}ds({\bm d})\right],
\en
where
\ben
{\bm d}\otimes {\bm d}:= {\bm d}{\bm d}^T = \left[
\begin{matrix}
d_1^2&d_1d_2\\
d_2d_1&d_2^2
\end{matrix}
\right],\quad {\bm d}=(d_1,d_2)^T\in\mathbb{S}.
\enn
Set ${\bm e}_1=(1,0)^{T}$ and ${\bm e}_2=(0,1)^{T}$. Then, for $j=1,2$,
\be \label{impi}
\I(\Pi({x,z}){\bm e}_j)=\frac{1}{8\pi}\left[\frac{1}{\la+2\mu}\int_{\mathbb{S}}d_j{\bm d}e^{ik_p({x-z})
\cdot{\bm d}}ds({\bm d})+\frac{1}{\mu}\int_{\mathbb{S}}d_j^{\perp} {\bm d}^{\perp}e^{ik_s({x-z})\cdot
{\bm d}}ds({\bm d})\right],
\en
where $d_j^{\perp}$ stands for the $j$-th component of ${\bm d}^{\perp}$, that is, $d_1^{\perp}=-d_2$,
$d_2^{\perp}=d_1$.

\begin{lemma}\label{lemma1}
For ${\bm d}\in \overline{\mathbb{S}}_+$, the scattered fields corresponding
to the incident plane waves ${\bm u}^{in}_{p}(x;{\bm d})$ and ${\bm u}^{in}_{s}(x;{\bm d})$
are given as ${\bm u}^{sc}_{p,+}(x;{\bm d}):=-{\bm d}e^{i{k_p}x\cdot {\bm d}}$
and ${\bm u}^{sc}_{s,+}(x;{\bm d}):=-{\bm d}^{\perp}e^{i{k_s}x\cdot {\bm d}}, x\in \Om$,
respectively.
\end{lemma}

\begin{proof}
It is easily seen that ${\bm u}^{sc}_{p,+}(x;{\bm d})$  satisfies the {\em Navier} equation (\ref{navier})
and the vertical growth rate condition defined in Problem \ref{pb1}. Since ${\bf u}^{sc}_{p,+}(x;{\bf d})$
is upward propagating for ${\bm d}\in \overline{\mathbb{S}}_+$, it satisfies
the UPRC in (\ref{uprc}); see \cite[Remark 2.14]{A00}. Further, we have the boundary data
${\bm u}^{sc}_{p,+}(x;{\bm d})=-{\bm u}^{in}_{p}(x;{\bm d})$ on $\G$. By the uniqueness of solutions to
the forward scattering problem (see Theorem \ref{thmwp}), we conclude that the function
${\bm u}^{sc}_{p,+}(x;{\bm d}):=-{\bm d}e^{i{k_p}x\cdot{\bm d}},\;x\in\Om$, is the unique scattered
field corresponding to the incident plane wave ${\bm u}^{in}_{p}(x;{\bm d})$.
The shear wave case can be proved similarly.
\end{proof}

Denote by ${\bm u}^{sc}_{p,-}(x;{\bm d}),\;{\bm u}^{sc}_{s,-}(x;{\bm d})\in[C^2(\Om)\cap C(\ov{\Om})\cap H^1_{loc}(\Om)]^2$
the unique scattered field corresponding to the incident plane waves ${\bm u}^{in}_p(x; {\bm d})$ and
${\bm u}^{in}_s(x; {\bm d})$ with ${\bm d}\in\mathbb{S}_-$, respectively.
Introduce new incident waves of the form
\be\label{uin}
{\bm U}^{in}~({ x;z},{\bm e}_j):=\text{Im}~(\Pi({ x,z}){\bm e}_j),\quad j=1,2,
\en
for $x, z\in\Omega$, $x\neq z$.
Next we shall express the scattered field incited by ${\bm U}^{in}~({ x;z},{\bm e}_j)$ in terms of the functions
${\bm u}^{sc}_{p,-}$ and ${\bm u}^{sc}_{s,-}$.

\begin{theorem}\label{mainthm}
The scattered field generated by ${\bm U}^{in}~({ x;z},{\bm e}_j)$
 takes the form
\begin{align}\no
{\bm U}^{sc}~({ x;z},{\bf e}_j)&=\frac{1}{8\pi}\left[\frac{1}{\lambda+2\mu}\int_{\mathbb{S}_-}
{\bm u}^{sc}_{p,-}(x;{\bm d})d_je^{-ik_pz\cdot {\bm d}}ds({\bm d})+\frac{1}{\mu}\int_{\mathbb{S}_-}
{\bm u}^{sc}_{s,-}(x;{\bm d})d_j^{\perp}e^{-ik_s{\bf z}\cdot {\bm d}}ds({\bm d})\right]\\ \label{usc}
&-\frac{1}{8\pi}\left[\frac{1}{\la+2\mu}\int_{\mathbb{S}_-}d_j'{\bm d}'e^{ik_p{(x'-z')}\cdot{\bm d}}ds({\bm d})
+\frac{1}{\mu}\int_{\mathbb{S}_-}({d}')^{\perp}_j({\bm d}')^{\perp}e^{ik_s({x'-z'})\cdot{\bm d}}ds({\bm d})\right].
\end{align}
Here, $x'=(x_1,-x_2)$ for $x=(x_1,x_2)\in \R^2$. The notation ${d}'_j$ and $({d}')^{\perp}_j$ denote
the $j$-th component of ${\bm d}'$ and $({\bm d}')^{\perp}$, respectively, given by
\ben
{d}'_1=d_1,\quad {d}'_2=-d_2,\quad ({d}')^{\perp}_1=d_2,\quad ({d}')^{\perp}_2=d_1.
\enn
\end{theorem}

\begin{proof}
In view of (\ref{impi}), the incident field  ${\bm U}^{in}~({ x;z},{\bm e}_j)$ can be decomposed into
the sum of four parts:
\ben
{\bm U}^{in}=\frac{1}{8\pi}\left\{\frac{1}{\la+2\mu}({\bm U}^{in}_{p,-}+{\bm U}^{in}_{p,+})
+\frac{1}{\mu}({\bm U}^{in}_{s,+}+{\bm U}^{in}_{s,-})\right\},
\enn
where
\ben
{\bm U}^{in}_{p,\pm}({x;z},{\bm e}_j)
:=\int_{\mathbb{S}_\pm}{\bm u}^{in}_{p}(x;{\bm d})d_je^{-ik_pz\cdot {\bm d}}ds({\bm d})
=\int_{\mathbb{S}_\pm}{\bm d} d_je^{ik_p(x-z)\cdot {\bm d}}ds({\bm d}),\\
{\bm U}^{in}_{s,\pm}({x;z},{\bm e}_j):=\int_{\mathbb{S}_\pm}{\bm u}^{in}_{s}(x;{\bm d})
d_j^{\perp}e^{-ik_sz\cdot {\bm d}}ds({\bm d})
=\int_{\mathbb{S}_\pm}{\bm d}^{\perp} d_j^{\perp}e^{ik_s(x-z)\cdot {\bm d}}ds({\bm d}).
\enn
By linear superposition and Lemma \ref{lemma1}, the unique scattered field ${\bm U}^{sc}_{p,+}$
that corresponds to ${\bm U}^{in}_{p,+}$ is given by
\ben
{\bm U}^{sc}_{p,+}({x;z},{\bm e_j})
:=\int_{\mathbb{S}_+}{\bm u}^{sc}_{p,+}(x;{\bm d})d_je^{-ik_pz\cdot {\bm d}}ds({\bm d})
=-\int_{\mathbb{S}_+} {\bm d}\,d_j e^{ik_p(x-z)\cdot {\bm d}}ds({\bm d}),
\enn
implying that
\ben
{\bm U}^{sc}_{p,+}({x;z},{\bm e_j})=-{\bm U}^{in}_{p,+}({x;z},{\bm e_j})
=-\int_{\mathbb{S}_-}d_j'{\bm d}'e^{ik_p{ (x'-z')}\cdot {\bm d}}ds({\bm d}).
\enn
Analogously, we have
\ben
{\bm U}^{sc}_{s,+}({x;z},{\bm e}_j)=-{\bm U}^{in}_{s,+}({x;z},{\bm e_j})=
\int_{\mathbb{S}_-}({d}')^{\perp}_j({\bm d}')^{\perp}e^{ik_s({x'-z'})\cdot {\bm d}}ds({\bm d}) .
\enn
On the other hands, it is easy to see that the unique scattered fields incited by ${\bm U}^{in}_{p,-}$ and
${\bm U}^{in}_{s,-}$ can be expressed as
\ben
{\bm U}^{sc}_{p,-}({x;z},{\bm e_j})
&:=&\int_{\mathbb{S}_-}{\bm u}^{sc}_{p,-}(x;{\bm d})d_je^{-ik_pz\cdot {\bm d}}ds({\bm d}),\\
{\bm U}^{sc}_{s,-}({x;z},{\bm e_j})
&:=&\int_{\mathbb{S}_-}{\bm u}^{sc}_{s,-}(x;{\bm d})d_je^{-ik_sz\cdot {\bm d}}ds({\bm d}),
\enn
respectively. To sum up, we may rewrite the scattered field incited by ${\bm U}^{in}$ as
\ben
{\bm U}^{sc}=\frac{1}{8\pi}\left\{\frac{1}{\la+2\mu}({\bm U}^{sc}_{p,-}+{\bm U}^{sc}_{p,+})
+\frac{1}{\mu}({\bm U}^{sc}_{s,+}+{\bm U}^{sc}_{s,-})\right\},
\enn
which takes the same form as the right-hand side of (\ref{usc}). The proof is thus completed.
\end{proof}

From the proof of Theorem \ref{thmwp}, we can represent ${\bm U}^{sc}(x;z,{\bm e}_j)$ in (\ref{usc})
as the layer-potential
\be\label{eq1}
{\bm U}^{sc}(x;z,{\bm e}_j)=\int_{\G}\left\{{\bm P}_y[\Pi_h(x,y)]-i\eta{\bm\Pi}_{h}({x,y})\right\} {\bm\psi}^{(j)}_z({y})ds({y}),\quad x\in\Om,
\en
where the density function ${\bm\psi}_z^{(j)}$ is the unique solution to the integral equation
\ben
({\bm I}+{\bm D}-i\eta {\bm S}){\bm\psi}_z^{(j)}=-2{\bm G}^{(j)}_z\quad\mbox{on}\quad \Gamma,
\enn with
\ben
{\bm G}^{(j)}_z(x):=-{\bm U}^{in}(x;z,{\bm e}_j)=-\text{Im}~\Pi(x;z){\bm e}_j.
\enn
Here, we use the subscript $z$ to indicate the dependence of $\psi_z^{(j)}$ on the point $z$.
By (\ref{Pi}) and a straightforward calculation it follows that
\be\label{ImP}
{\rm Im}[\Pi_{j,k}(x,z)]=\frac{1}{4\mu}\left[F_1(|x-z|)\,\delta_{j,k}+F_2(|x-z|)\,
\frac{(x_j-z_j)(x_k-z_k)}{|x-z|^2}\right],
\en
where
\ben
F_1(t)&=&J_0(k_st)-\frac{1}{k_st}\left(J_1(k_st)-\frac{k_p}{k_s}\; J_1(k_pt)\right),\\
F_2(t)&=&\frac{2}{k_st} J_1(k_st)-J_0(k_st)-\frac{k_p}{t} J_1(k_pt)+\frac{k_p^2}{k_s^2} J_0(k_pt).
\enn

We remark that the Bessel functions have the following behavior \cite[Section 2.4]{colton-kress-2013}
(see also Figure \ref{fig10})
\ben
J_n(t)=\sum_{p=0}^\infty\frac{(-1)^p}{p!(n+p)!}\left(\frac{t}{2}\right)^{n+2p},\quad t\in\R.
\enn
For large arguments, it holds that
\ben
J_n(t)=\sqrt{\frac{2}{\pi t}}\cos\left(t-\frac{n\pi}{2}-\frac{\pi}{4}\right)\left\{1+O\left(\frac{1}{t}\right)\right\},
\quad t\rightarrow\infty.
\enn
\begin{figure}[htbp]
\centering
  \includegraphics[width=3in]{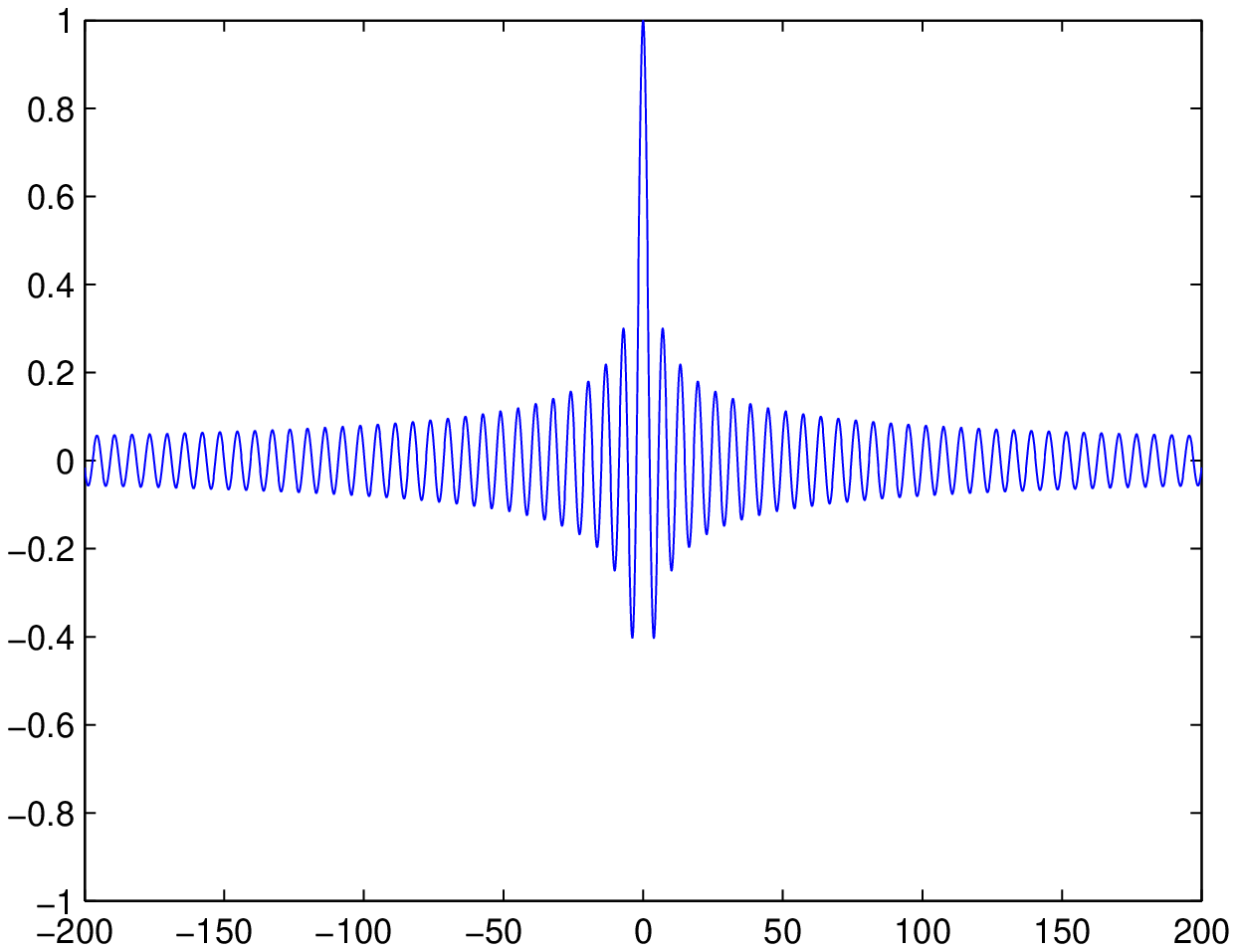}
  \includegraphics[width=3in]{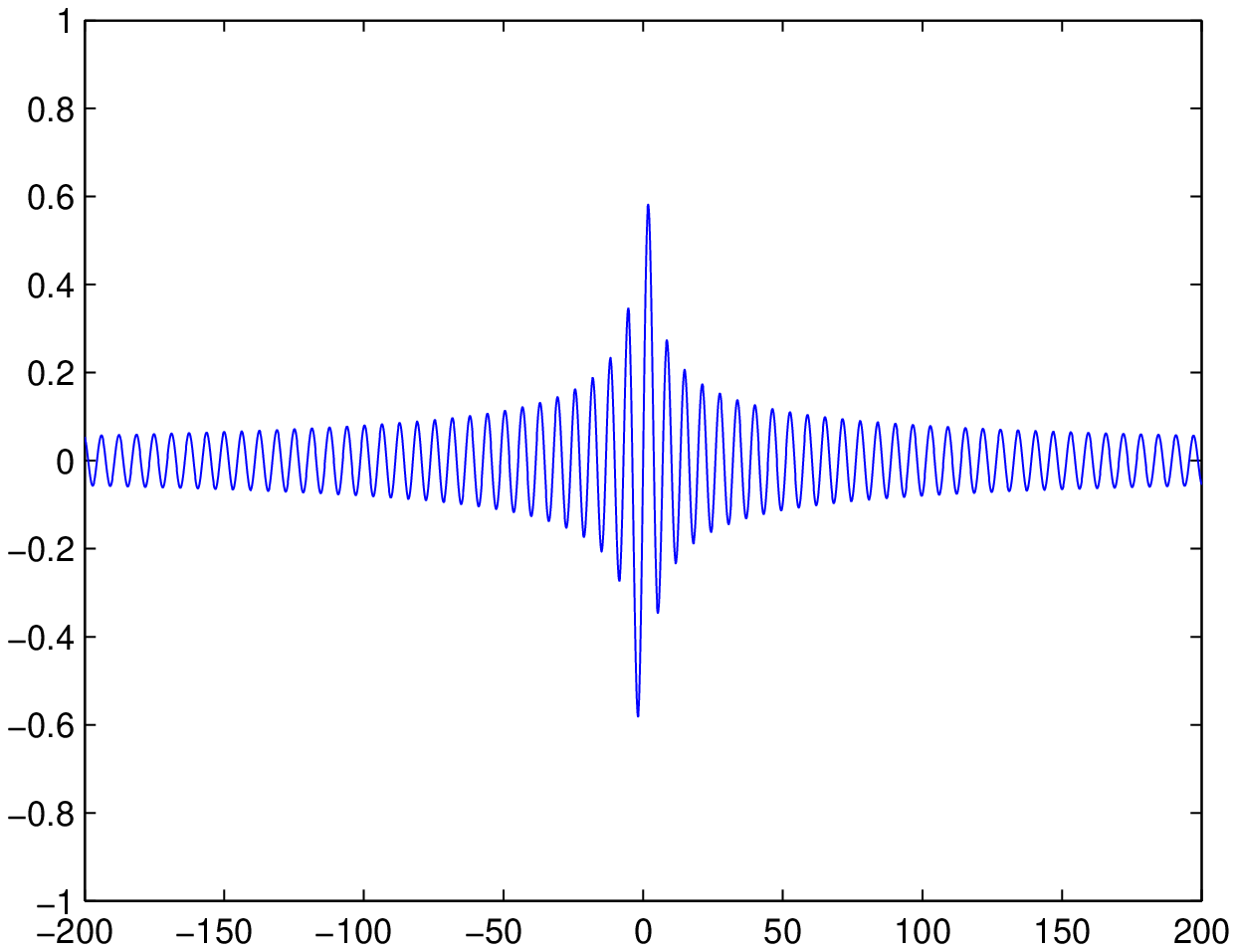}
\caption{The behavior of the Bessel functions $J_0$ (left) and $J_1$ (right).}
\label{fig10}
\end{figure}
Thus, from the expression of $\I[\Pi(x,z)]$ in (\ref{ImP}) we have the estimate
\ben
\max\limits_{j,k=1,2} |\I\Pi_{jk}({x,z})|=\left\{
\begin{array}{ll}
O(1) &\textrm{as}\;|x-z|\rightarrow 0,\\
O\left({|{ x-z}|^{-1/2}}\right)&\textrm{if}\;\;|{x-z}|\rightarrow \infty,
\end{array} \right.
\; x\in \Gamma.
\enn
Further, since ${\bm I}+{\bm D} - i \eta {\bm S}$ is bijective (and so boundedly invertible)
in $[BC(\Gamma)]^2$, it holds that
\be\label{phi}
C_1 \|{\bm G}^{(j)}_{z}\|_{\infty,\Gamma}
\leq \|{{\bm \psi}^{(j)}_{z}}\|_{\infty,\Gamma}
\leq C_2 \|{\bm G}^{(j)}_{z}\|_{\infty,\Gamma},\quad j=1,2
\en
for some positive constants $C_1$ and $C_2$. By (\ref{phi}) we have
\ben
\left\{\begin{array}{ll}
 \|{{\bm \psi}^{(j)}_{z}}\|_{\infty,\Gamma}
\geq C>0 &\textrm{if}\;\;\mbox{dist}(z, \G)\rightarrow 0,\\
 \|{{\bm \psi}^{(j)}_{z}}\|_{\infty,\Gamma}
=O\left({d({\bf z},\Gamma)^{-1/2}}\right)&\textrm{as}\;\mbox{dist}(z, \G)\rightarrow \infty,
\end{array} \right.\;j=1,2.
\enn
Combining the above estimate with (\ref{eq1}), we expect that the scattered field
${\bm U}^{sc}({x;z},{\bm e}_j)$ for $x\in\Omega$ takes a relatively large value when the
sampling point $z$ is getting close to the rough surface $\G$ and decays with the order
 $1/\mbox{dist}(z,\G)$ as ${z}$ moves away from rough surface $\G$.
Motivated by the above discussions, we propose the imaging function
\be\label{indicator}
I({z}):=\sum^{2}_{j=1}\int_{\G_{a}}\big|{\bm U}^{sc}({x;z},{\bm e}_j)\big|^2ds({x}),
\en
for some $a>f^+$, where $z\in\R^2$ is the sampling point in a searching region.
Analogously, it is reasonable to expect that the imaging function $I(z)$ decays as $z$ moves
away from the rough surface $\G$. Hence, $I(z)$ can regarded as an imaging function for recovering $\Gamma$.

In our numerical computations, the straight line $\G_a$ in (\ref{indicator}) is truncated by a finite line segment
${\G_{a,A}}:=\{x\in\Gamma_a~|~|x_1|<A\}$, which will be discretized uniformly into $2N$ subintervals with
the step size is $h=A/N.$ In addition, the lower-half circle $\mathbb{S}_-$
in (\ref{usc}) will also be uniformly discretized into $M$ grids with the step size $\Delta\theta=\pi/M$.
Then for each sampling point $z$ we get the discrete form of (\ref{indicator}) as follows:
\be\no
&&I_A(z)=\sum_{j=1}^{2}\left|h\sum_{i=0}^{2N}
\frac{\Delta\theta}{8\pi}\sum_{k=0}^{M}\left(
\frac{1}{\lambda+2\mu}{\bm u}^s_p({x_i},{\bm d}_k)d_je^{-ik_pz\cdot {\bm d}_k}
+\frac{1}{\mu}{\bm u}^s_s({x_j},{\bm d}_k)d_j^{\perp}e^{-ik_s{z}\cdot {\bm d}_k}\right.\right.\\ \label{eq188}
&&\qquad\qquad\qquad\qquad\left.\left.
-\frac{1}{\lambda+2\mu}(d_k)_j'{\bm d}_k'e^{ik_p{(x_j'-z')}\cdot {\bm d}_k}
-\frac{1}{\mu}(d_k')^{\perp}_j(d_k')^{\perp}e^{ik_s({x_j'-z'})\cdot {\bm d}_k}
\right)\right|^2.
\en
Here, the measurement positions are denoted by $x_j=(-A+jh,H),$ $j=0,1,...,2N,$
and the incident directions ${\bm d}_k=(\sin(-\pi+k\Delta\theta),\cos(-\pi+k\Delta\theta)),$ $k=0,1,\cdots,M.$

Let $K\subset \R^2$ be a sampling region which contains part of
the rough surface to be recovered. The direct imaging scheme (\ref{eq188}) can be implemented as follows.

\begin{algorithm}
\begin{enumerate}
\item Choose $\mathcal{T}_m$ to be a mesh of $K$
and choose $\Gamma_{a,A}$ ($a>f_+$) to be a straight line segment above the rough surface.
\item Collect the scattered near-field data
      ${\bm u}^{sc}_p({x_j};{\bm d}_k)$ and ${\bm u}^{sc}_s({x_j};{\bm d}_k)$ for $x_j\in\G_{a,A}$, $j=0,\cdots,2N$, corresponding to the incident plane waves
      ${\bm u}^{in}_p({x};{\bm d}_k)$ and ${\bm u}^{in}_s({x};{\bm d}_k)$ with $k=0,\cdots,M$, respectively.
\item For each sampling point $z\in\mathcal{T}_m$, compute the imaging function $I_A(z)$ in (\ref{eq188}).
\item Plot the imaging function $I_A(z)$ for $z\in\mathcal{T}_m$ , where the large values represent the part of the rough surface in the sampling region $K$.
\end{enumerate}
\end{algorithm}

\section{Numerical examples}\label{sec4}
\setcounter{equation}{0}

In this section, we present several numerical experiments to demonstrate the effectiveness of
our imaging method. Emphasis will be placed upon the sensitivity of our inversion scheme to
the parameters involved, such as incident frequencies, length and height of the measurement line segment,
noisy levels and polarization directions.
We use the Nystr\"{o}m method to solve the forward elastic scattering problem for a rigid rough
surface \cite{LSZR16,LLZZ}. The scattered near-field data will be polluted by
\ben
{\bm u}^{sc}_{\delta}(x)={\bm u}^{sc}(x)+\delta(\zeta_1+i\zeta_2)\max_x|{\bm u}^{sc}(x)|,
\enn
where $\delta$ is the noise ratio and $\zeta_1,\zeta_2$ are standard normal distributions.
In all examples, we choose $N=200$ and $M=256$. The sampling region will be set to be a rectangular domain
and the Lam\'e constants are taken as $\mu=1, \lambda=1$.
In each figure, we use a solid line to represent the actual
rough surface against the reconstructed one.

In the first example, the rough surface is given by the function (see Figure \ref{fig2}):
$$
\;\;f_1(x_1) = \left\{
\begin{aligned}
&0.42-0.1\cos(0.75x_1)-0.05\cos(7x_1) &\quad x_1<4,\\
&0.55 &\quad \text{else}.
\end{aligned}
\right.
$$
The incident plane waves are incited at different frequencies and
the scattered near-field data are measured on $\Gamma_{a,A}$ with $a=2,A=8$ for each frequency.
Figure \ref{fig2} presents the reconstructed surfaces from the data with $20\%$ noise at
the frequencies $\om=15,20,25$, respectively.
It can be seen that the macro-scale features of the rough surface
are captured with a higher frequency $\om=25$ and the whole rough surface is accurately
recovered with a lower frequency $\om=15$.

\begin{figure}[htbp]
  \centering
  \subfigure[\textbf{\bm$\om=15$}]{\includegraphics[width=2in]{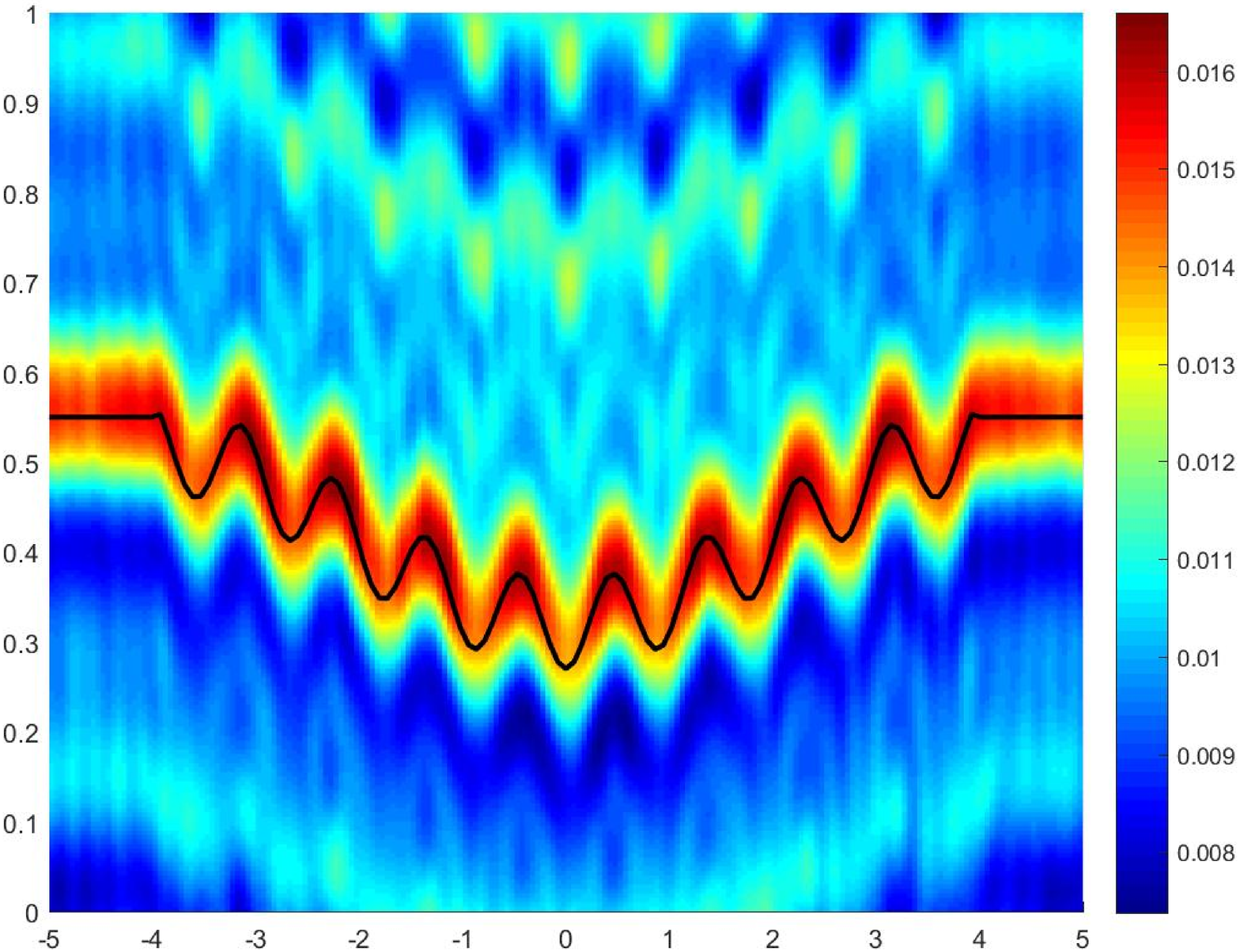}}
  \subfigure[\textbf{\bm$\om=20$}]{\includegraphics[width=2in]{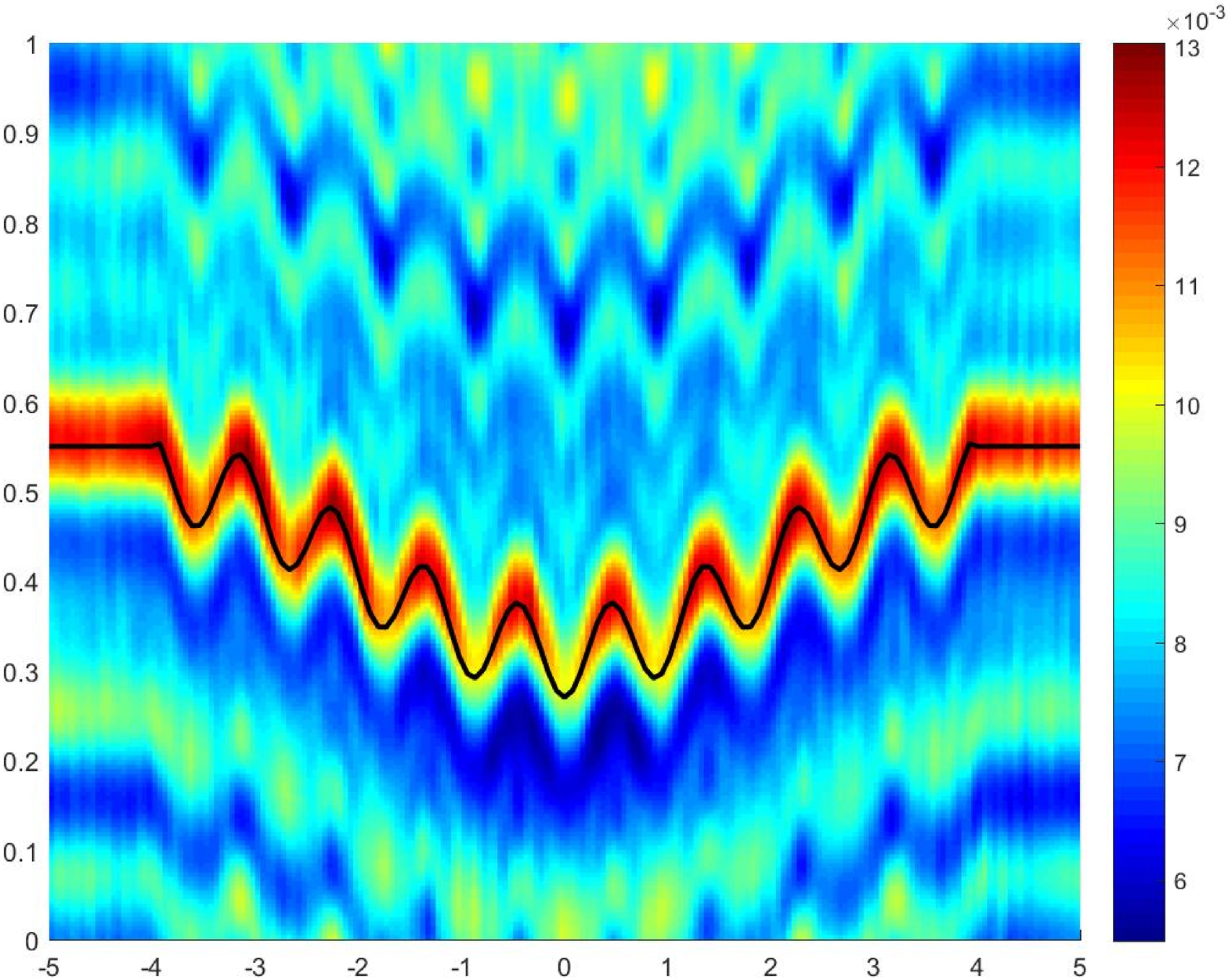}}
  \subfigure[\textbf{\bm$\om=25$}]{\includegraphics[width=2in]{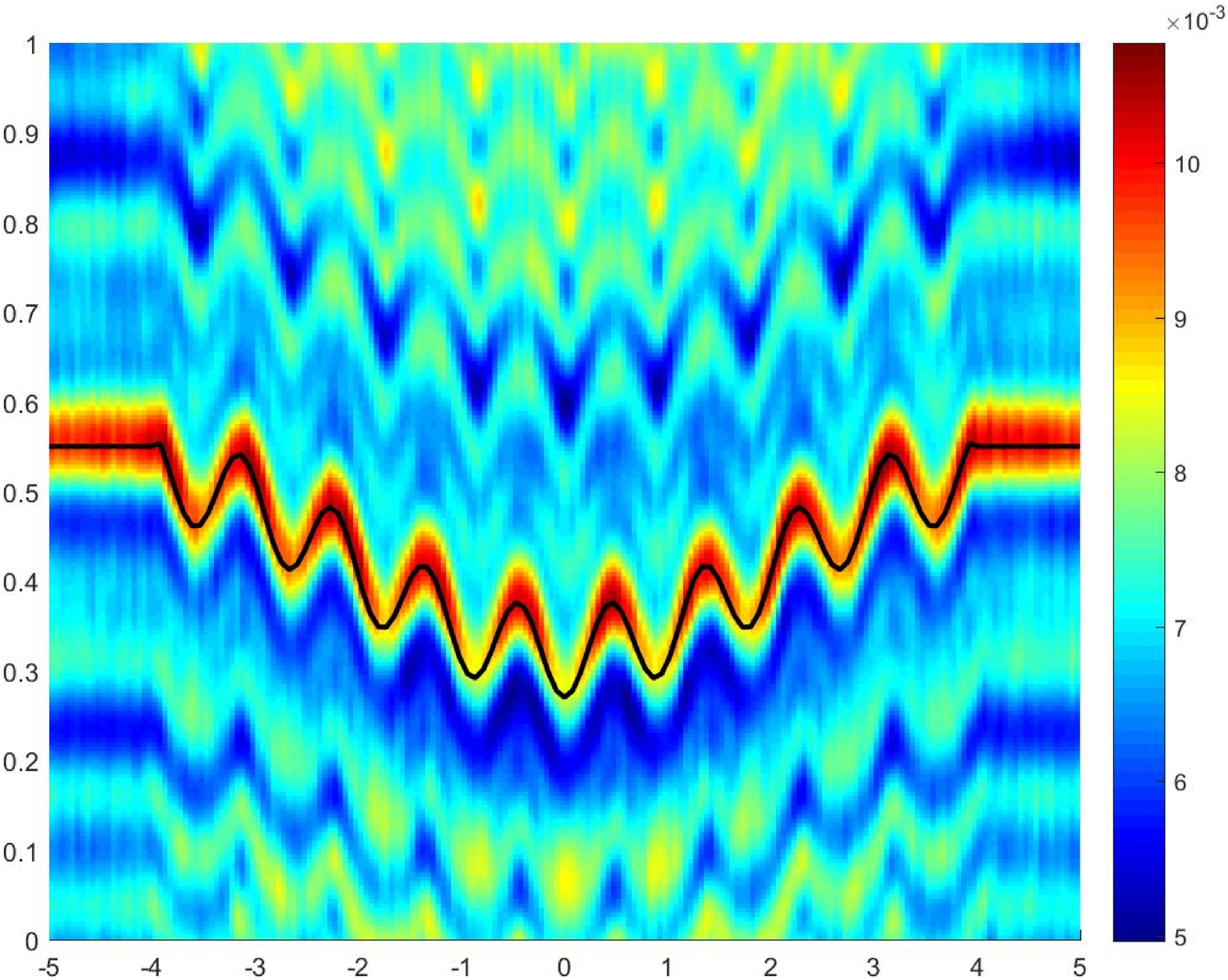}}
\caption{Reconstruction of a rough surface at different frequencies. }
\label{fig2}
\end{figure}

 In the second example, we use the data measured at $\Gamma_{a,A}$ with different parameters $a$ nd $A$.
 Note that $a>f_+$ and $2A>0$ denote respectively the hight and length of the line segment $\Gamma_{a,A}$.
 We fix the incident frequency to be $\om=20$ and the noise level to be 20\%.
 In Figure \ref{fig3}, the original rough surface is parameterized by
\ben
\;\;f_2(x_1)= 0.5+0.14\sin(0.7\pi (x_1+0.6)),
\enn
and the scattered near-field data are measured on $\G_{a,A}$ with $A=8$ and
$a=1.1, 2.0, 2.9$, respectively.
In Figure \ref{fig4},
the rough surface  takes the form
\ben
\;\;f_3(x_1)= 0.5+0.16\sin(\pi x_1)+0.1\sin(0.5\pi x_1),
\enn
with the scattered near-field data taken on $\G_{a,A}$ with
$a=2$ and $A=5,8,11$, respectively.
 One may conclude from Figures \ref{fig3} and \ref{fig4} that
the resolution is getting better if
the measurement surface $\G_{a,A}$ is getting closer to the rough surface or the length of $\G_{a,A}$ is getting longer.

\begin{figure}[htbp]
\centering
\subfigure[\textbf{$\{(x_1,1.1)~|~|x_1|<8)\}$}]{\includegraphics[width=2in]{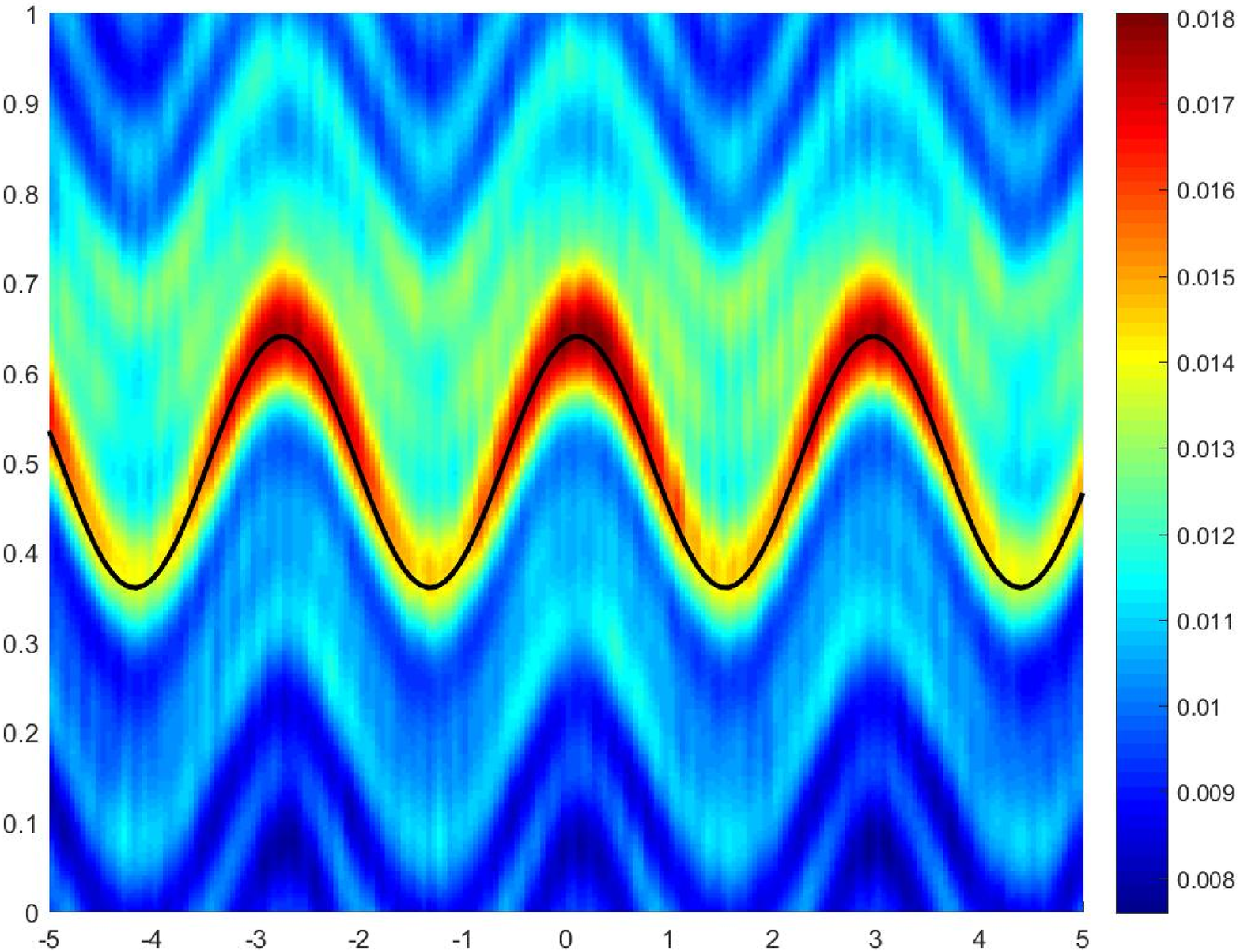}}
\subfigure[\textbf{$\{(x_1,2)~|~|x_1|<8)\}$}]{\includegraphics[width=2in]{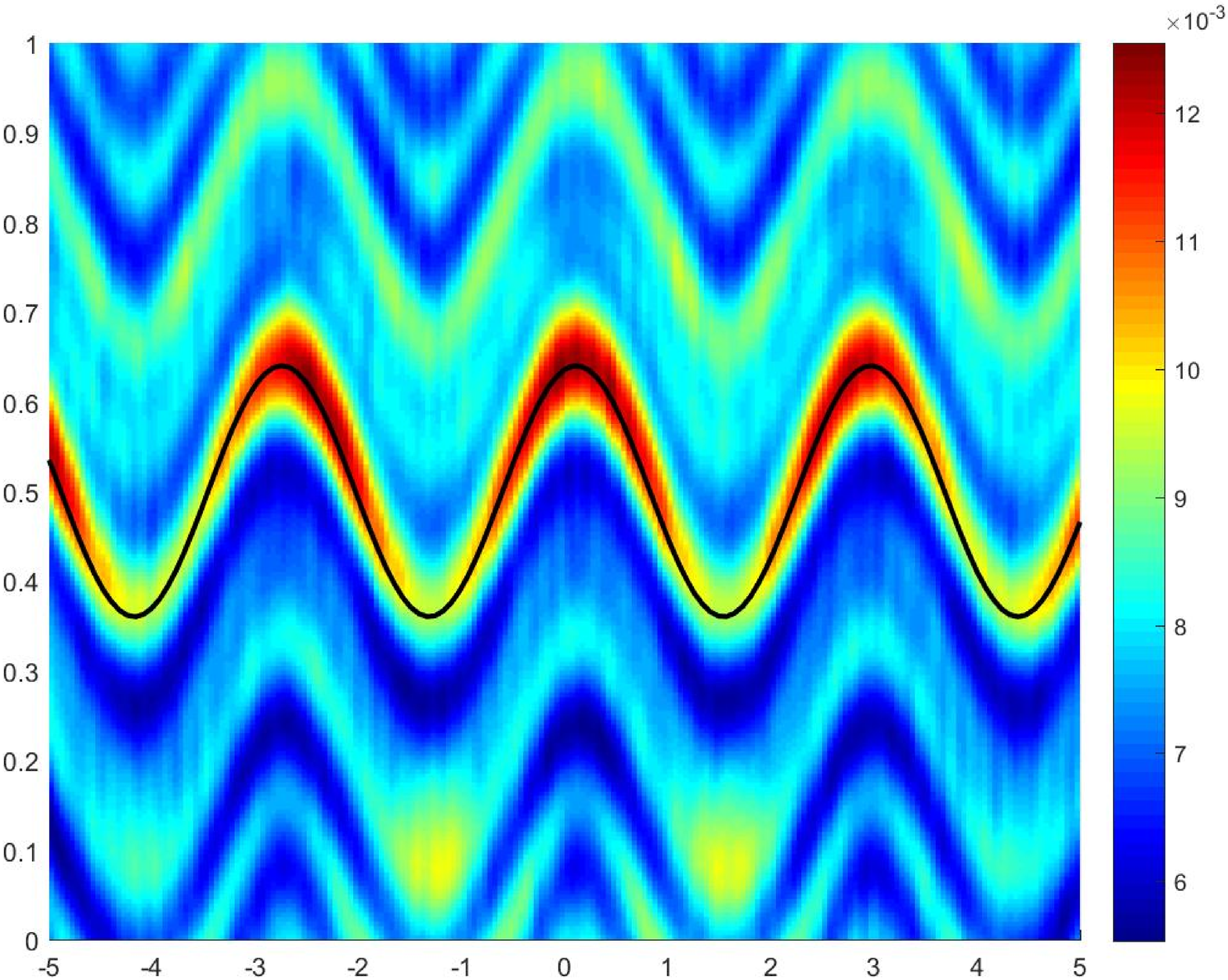}}
\subfigure[\textbf{$\{(x_1,2.9)~|~|x_1|<8)\}$}]{\includegraphics[width=2in]{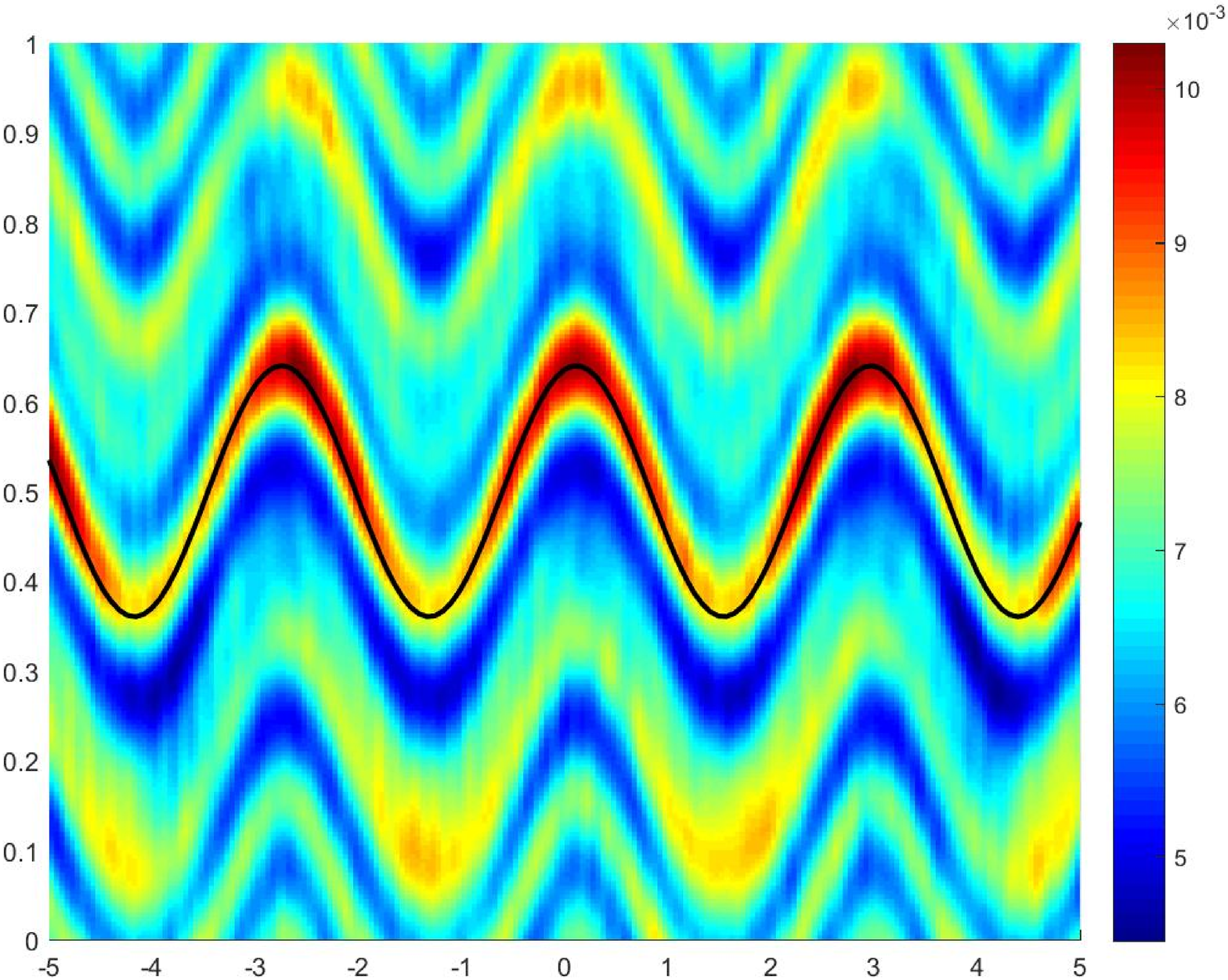}}
\caption{Reconstruction of a rough surface at different measurement heights.}
\label{fig3}
\end{figure}

\begin{figure}[htbp]
  \centering
\subfigure[\textbf{$\{(x_1,2)~|~|x_1|<5)\}$}]{\includegraphics[width=2in]{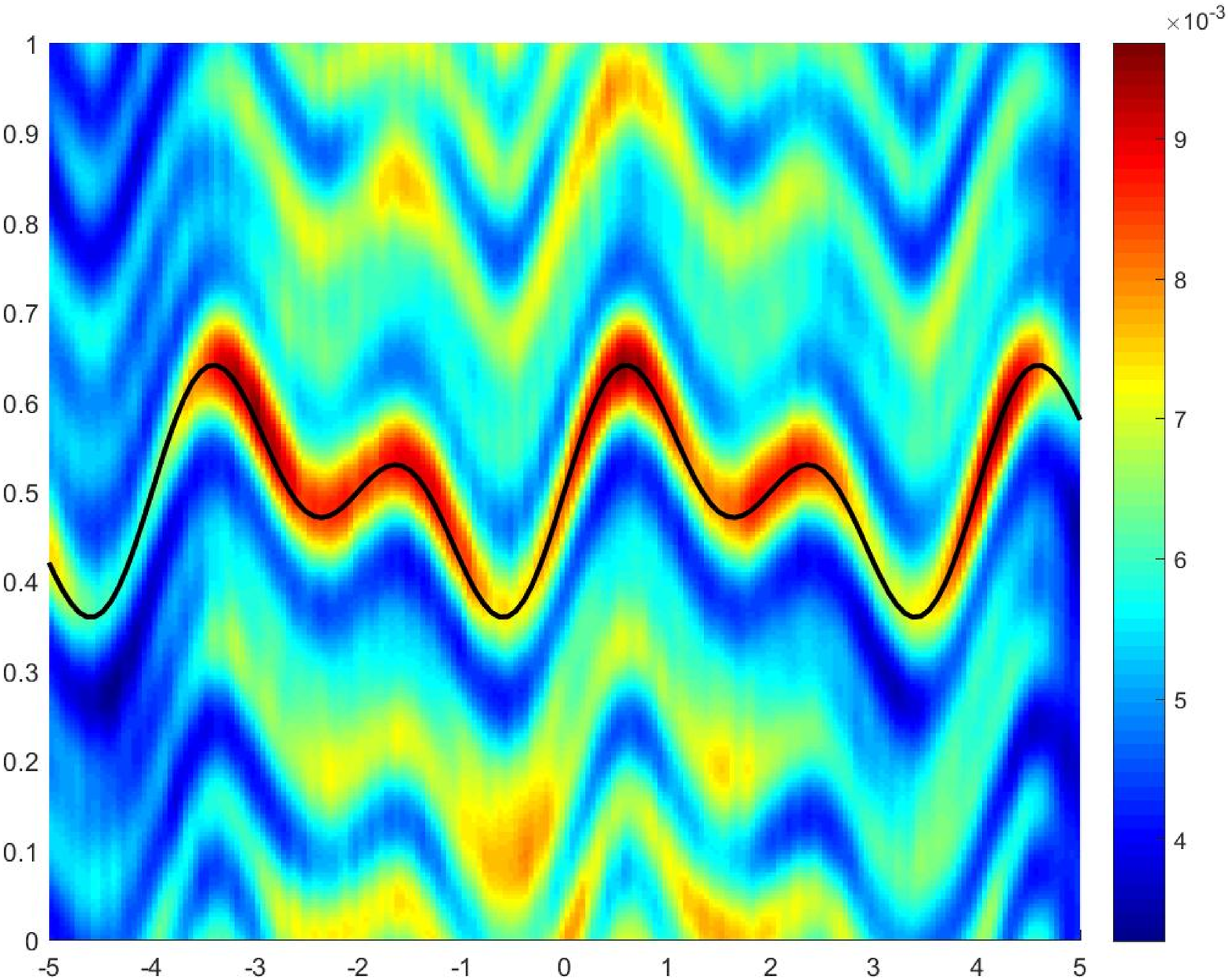}}
\subfigure[\textbf{$\{(x_1,2)~|~|x_1|<8)\}$}]{\includegraphics[width=2in]{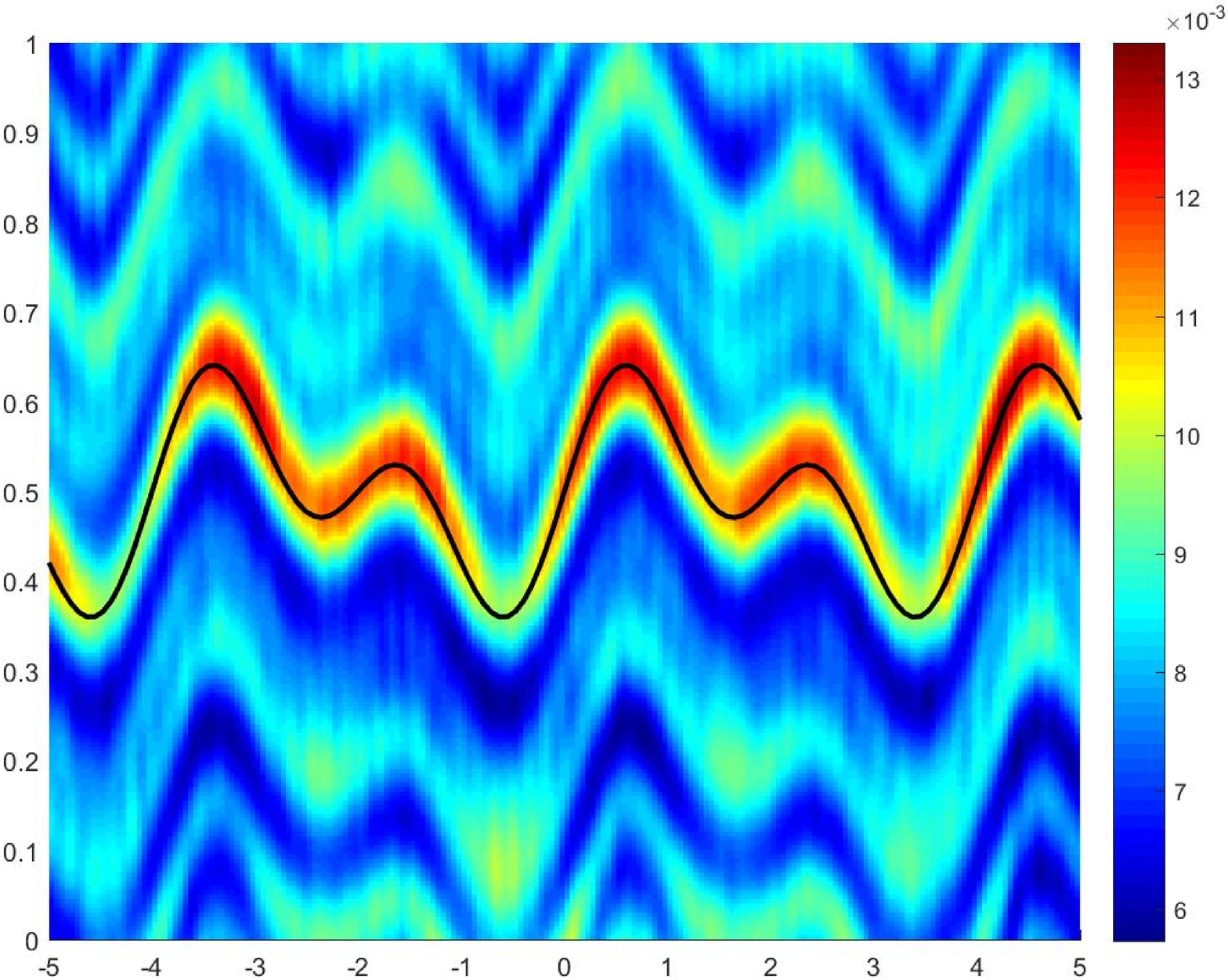}}
\subfigure[\textbf{$\{(x_1,2)~|~|x_1|<11)\}$}]{\includegraphics[width=2in]{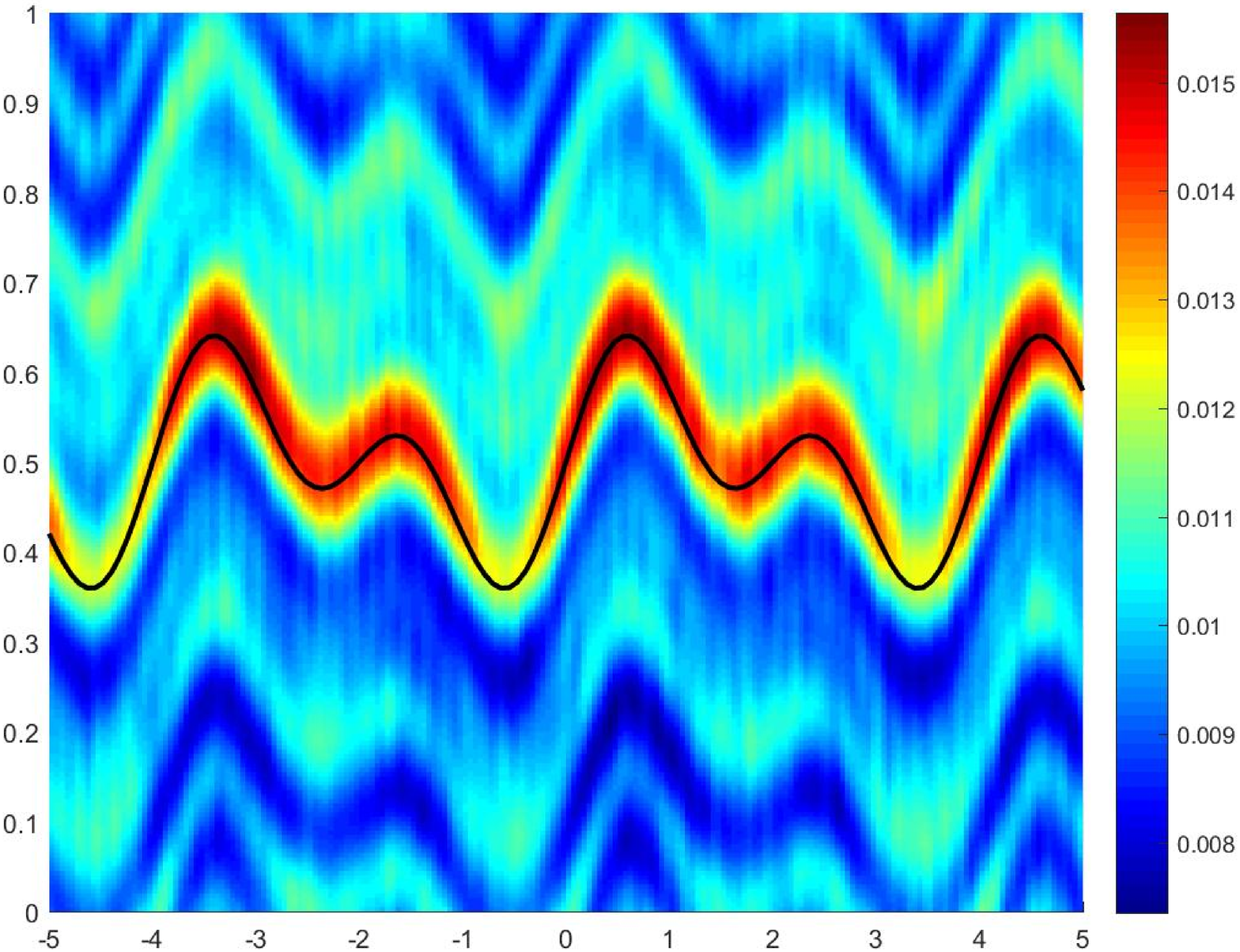}}
\caption{Reconstruction of a rough surface with different lengths of measured places.}
\label{fig4}
\end{figure}

In the third example, we make use of polluted data at different noise levels.
The scattered near-field data are taken on $\Gamma_{a,A}$ with $h=2,A=8$. The rough surface $\G_4$ is given by
\ben
\;\;f_4(x_1)= 0.5+0.084\sin(0.6\pi x_1)+0.084\sin(0.48\pi x_1)+0.03\sin(1.5\pi (x_1-1)).
\enn
It is shown Figure \ref{fig1} that the proposed scheme is robust to noise, even at the level of 40\% noise.

\begin{figure}[htbp]
  \centering
  \subfigure[\textbf{No noise}]{\includegraphics[width=2in]{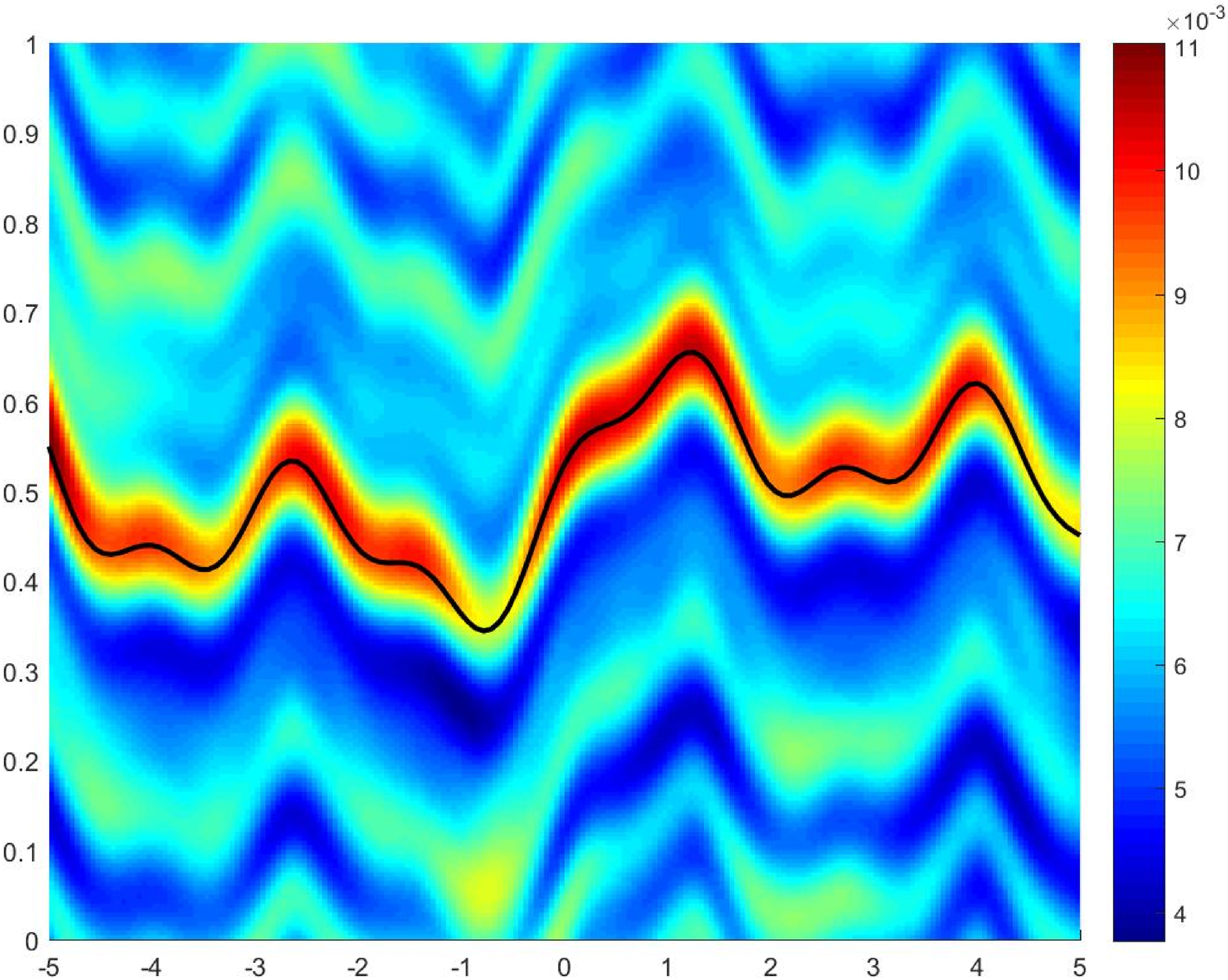}}
  \subfigure[\textbf{20\% noise}]{\includegraphics[width=2in]{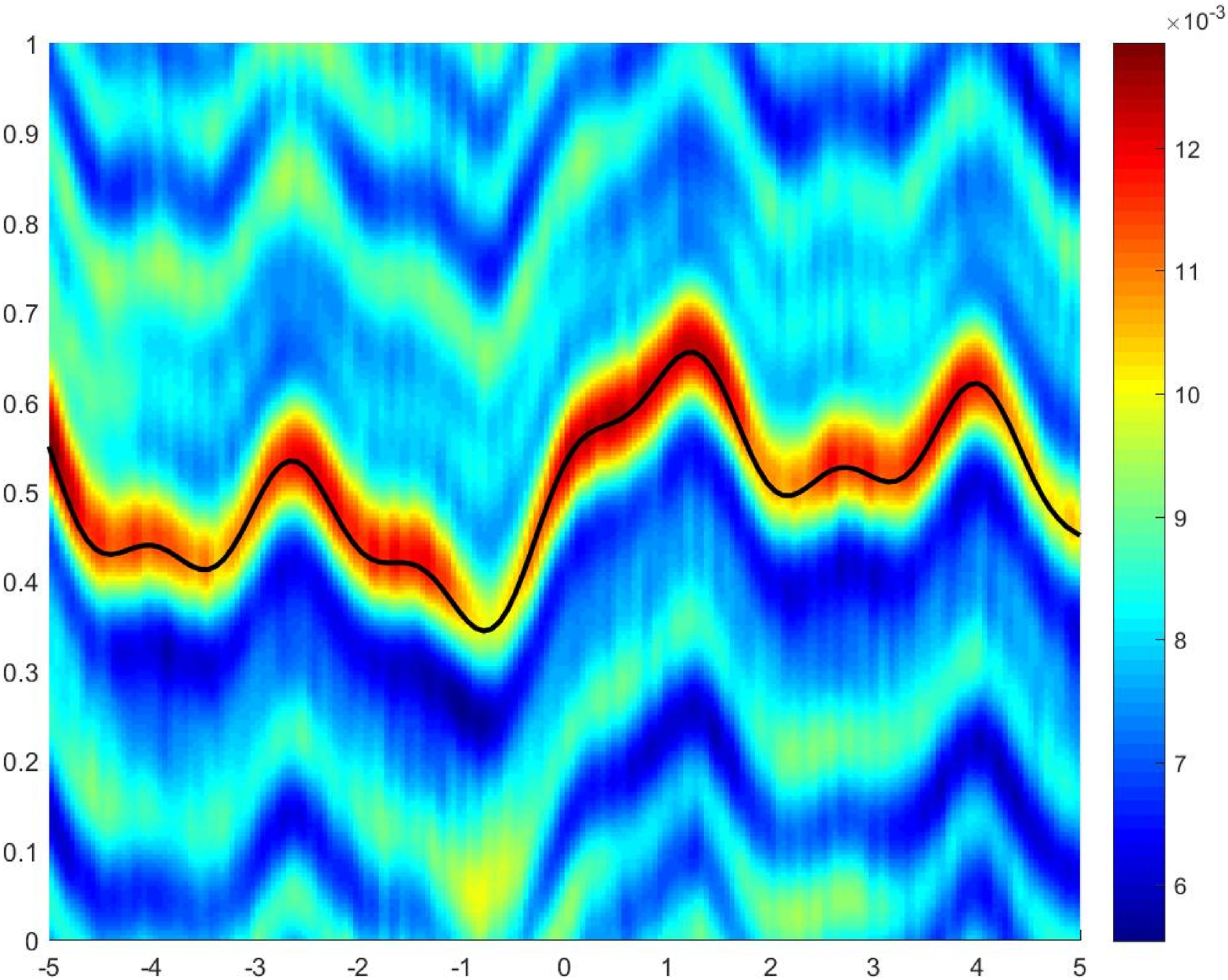}}
  \subfigure[\textbf{40\% noise}]{\includegraphics[width=2in]{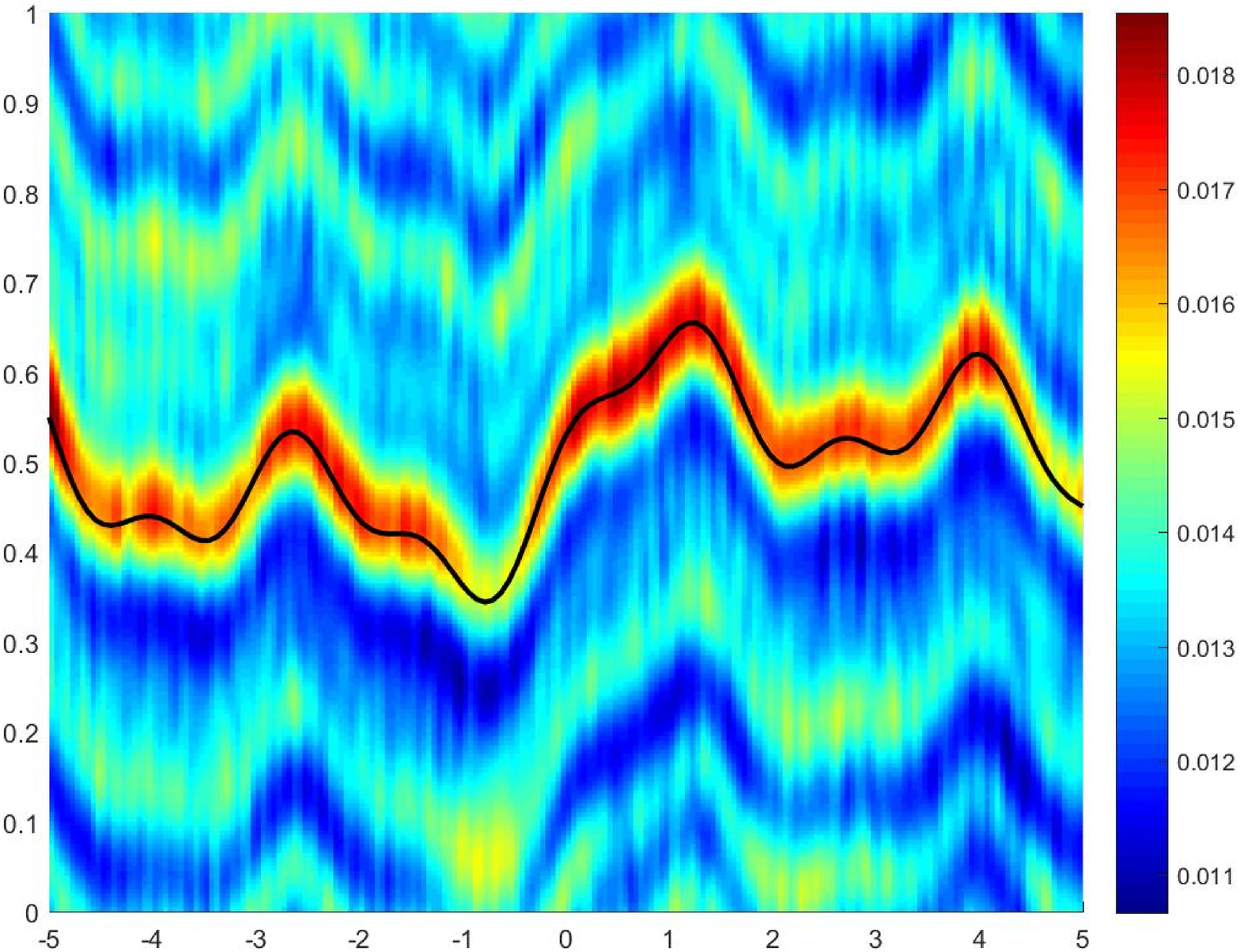}}
\caption{Reconstruction of a rough surface from data at different noise levels. }
\label{fig1}
\end{figure}

Finally, we compare the reconstructed solutions using different polarization directions ${\bm e}_j$ ($j=1,2$)
appeared in (\ref{uin}). Figure \ref{fig5} shows the results using
(a) the polarization ${p_1}= {\bm e}_1$;
(b) the polarization ${p_2}= {\bm e}_2$;
(c) the indicator function (\ref{eq188}).
The data are measured on $\Gamma_{a,A}$ with $a=2,A=8$.
We conclude that polarization $p_1$ can capture the rough surface accurately
but with some sidelobes, while the polarization $p_2$ can only find the convex part of the surface
but with few sidelobes. If we combine them together, that is, using our imaging function (\ref{indicator}),
the result can be improved since it inherits the advantages of the two polarizations.

\begin{figure}[htbp]
  \centering
  \subfigure[\textbf{$p_1={\bm e_1}.$ }]{
    \includegraphics[width=2in]{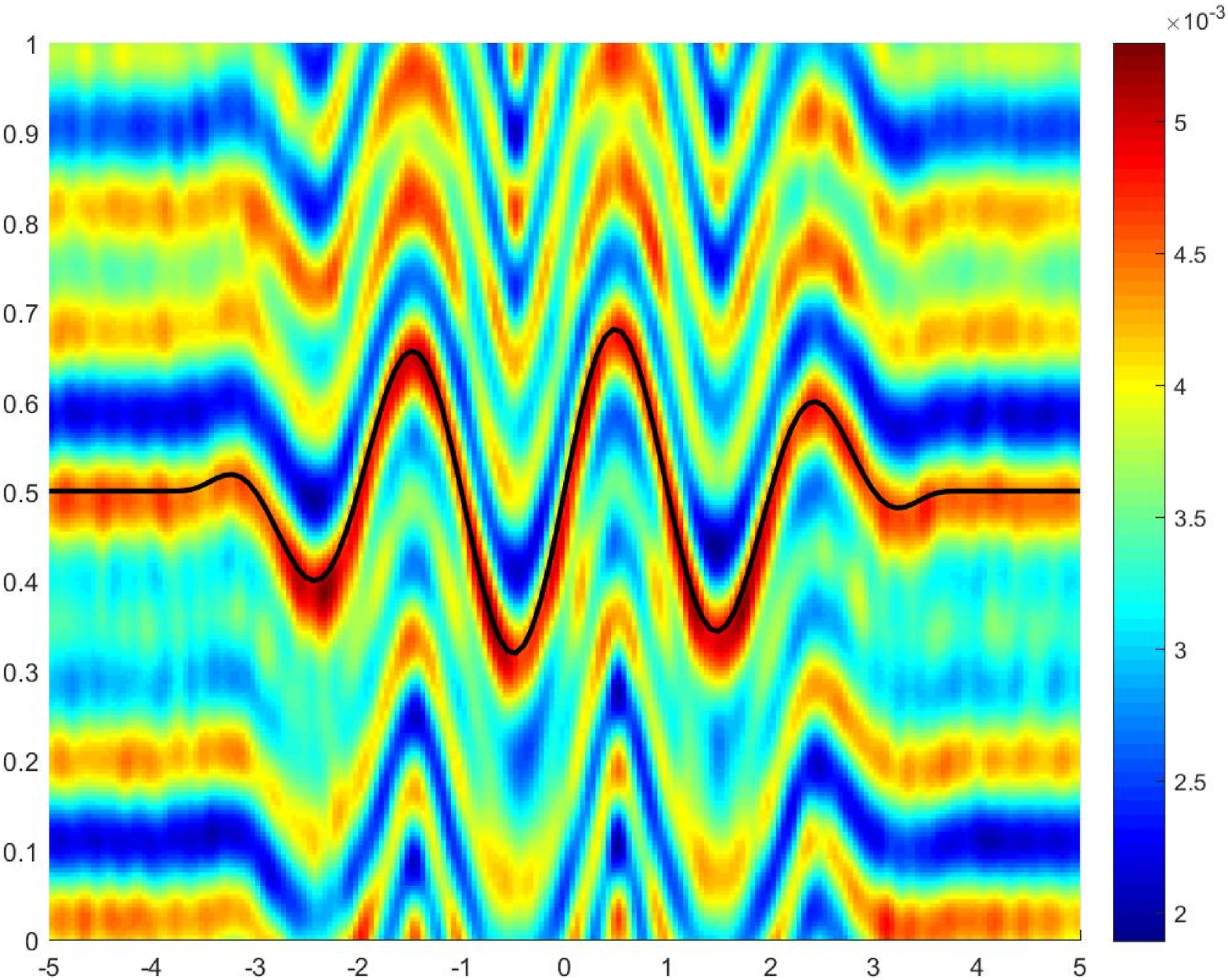}}
  \subfigure[\textbf{$p_2={\bm e_2}.$ }]{
    \includegraphics[width=2in]{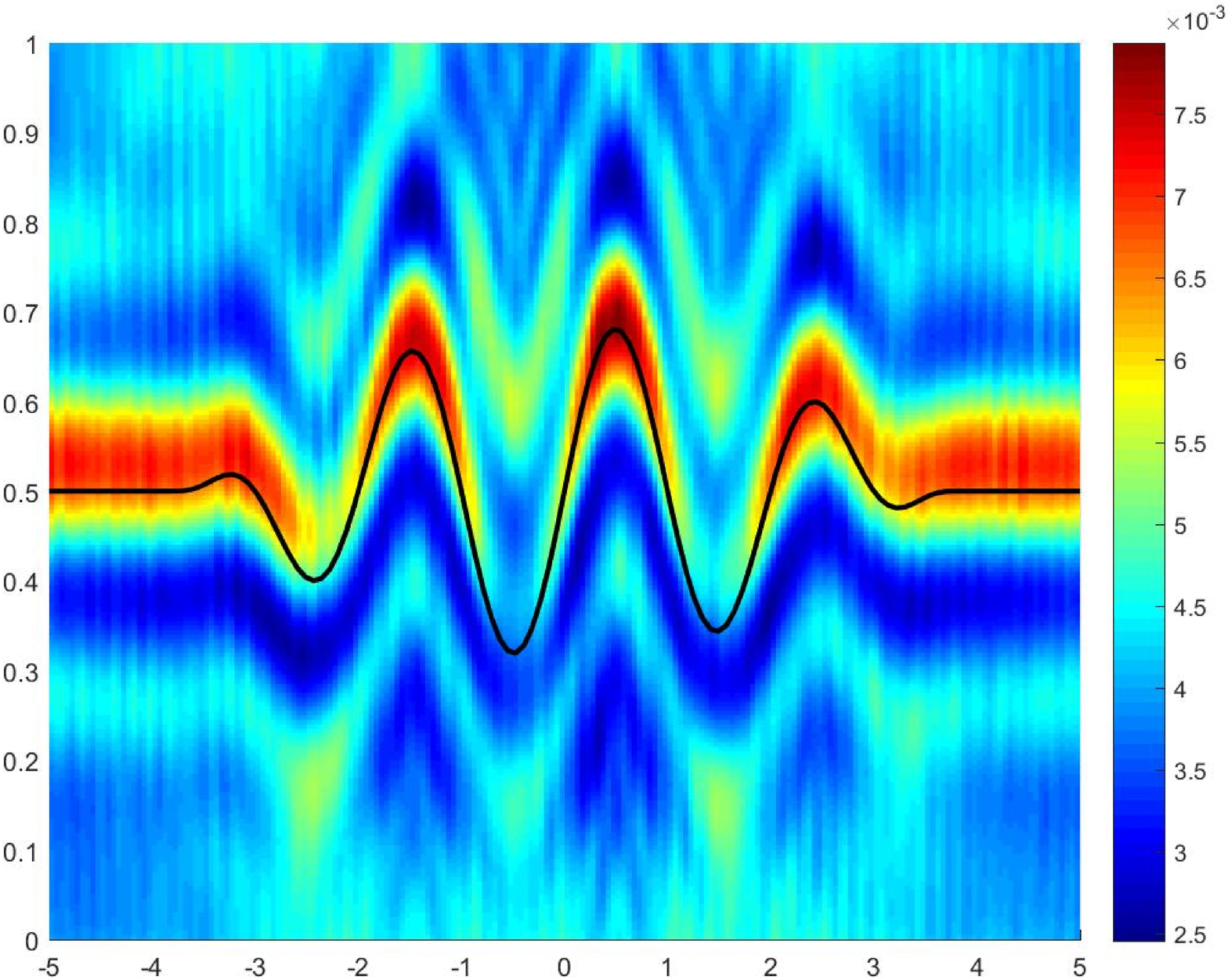}}
  \subfigure[\textbf{$p_1+p_2$ }]{
    \includegraphics[width=2in]{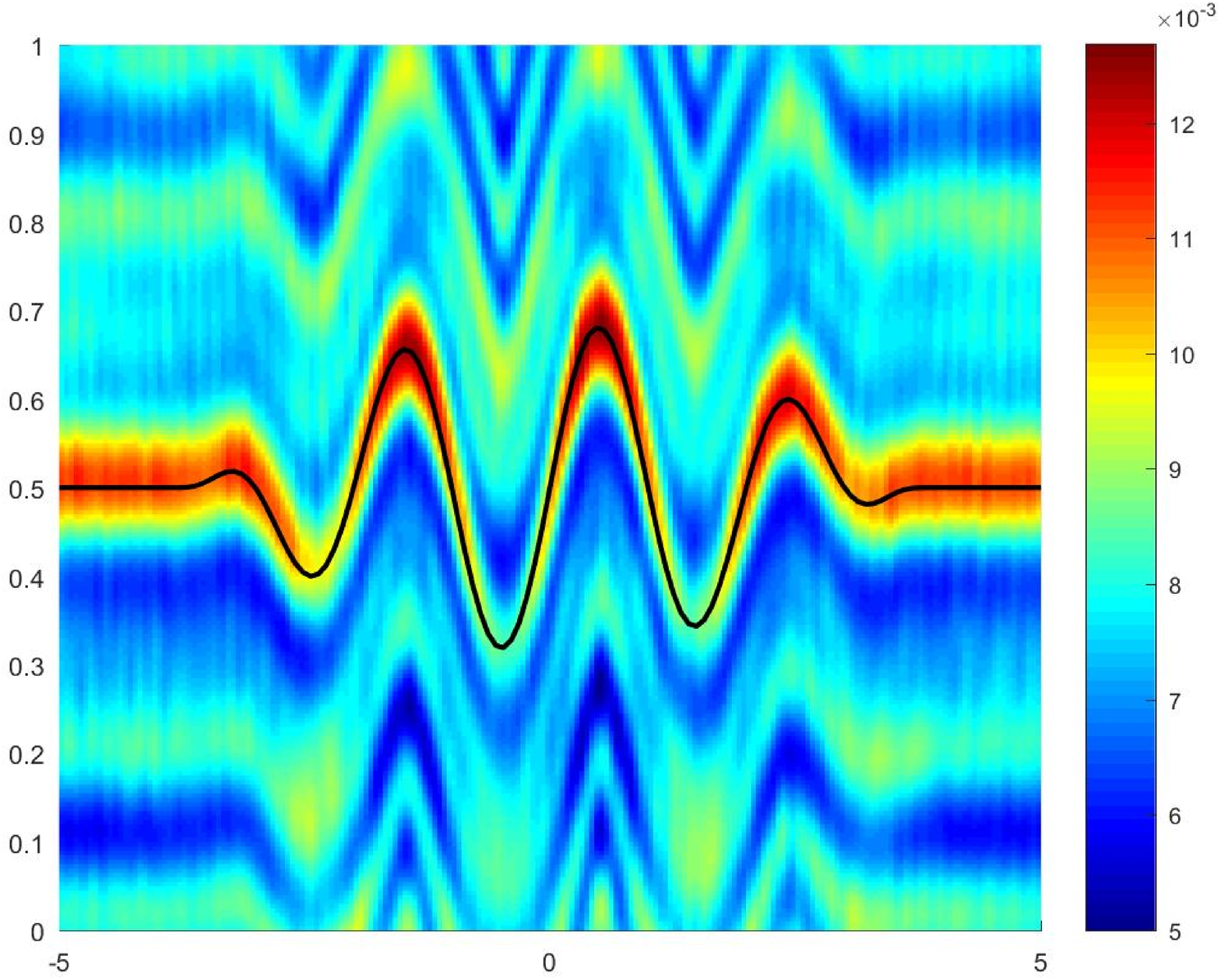}}
\caption{
Imaging results with different polarizations. The Figure (c) (right) is obtained through our indicator function for which the polarization directions ${\bm e}_1$ and ${\bm e}_2$ are both used.
}\label{fig5}
\end{figure}

The above numerical examples illustrate that the direct imaging method gives an accurate and stable reconstruction
of the unbounded rigid rough surface. In particular, the imaging algorithm is very robust to noise data.

\section{Conclusion}

We proposed a direct imaging method for recovering unbounded rigid rough surfaces from near-field data in linear
elasticity. Thanks to the Funk-Hecke formula and the free-field Green's tensor for the Navier equation,
we proposed the imaging function to reconstruct an unbounded rigid rough surface. The imaging function can be
easily implemented since only the calculation of inner products is involved.
Numerical experiments have been carried out to show that the reconstructions are accurate and robust to noise.
Further, the direct imaging algorithm could be extended to many other cases such as inverse
electromagnetic scattering problems by unbounded rough surfaces. Progress in these directions will be reported in a forthcoming paper.

\section*{Acknowledgements}

This work is partly supported by the NSFC of China grants 91630309, 11501558, 11571355 and 11671028,
the NSAF grant U1530401.


\end{document}